\documentclass[12pt,reqno]{amsart}
\usepackage{a4wide}
\usepackage{amssymb,amsmath,amsthm,newlfont,enumerate}
\newtheorem{theor}{Theorem }
\newtheorem{prop}{Proposition}[section]
\newtheorem{theo}[prop]{Theorem}
\newtheorem{lem}[prop]{Lemma}

\newcommand{\D}{\mathbb{D}}

\newcommand{\cM}{{\mathcal{M}}}

\newcommand{\Lat}{\mathrm{Lat}}

\newcommand{\Hol}{\mathrm{Hol}}

\newcommand\CC{\mathbb{C}}

\newcommand\TT{\mathbb{T}}
\newcommand\DD{\mathbb{D}}
\newcommand\NN{\mathbb{N}}

\DeclareMathOperator{\supp}{supp}

\newcommand{\cZ}{{\mathcal{Z}}}
\newcommand{\cJ}{{\mathcal{J}}}

\newcommand{\cD}{{\mathcal{D}}}

\newcommand{\cH}{{\mathcal{H}}}
\newcommand{\cP}{{\mathcal{P}}}
\newcommand{\cC}{{\mathcal{C}}}

\DeclareMathOperator{\dist}{dist}

\title[Cyclicity and  invariant subspaces in Dirichlet  spaces]
{Cyclicity and invariant subspaces in Dirichlet  spaces}

\author[O. El-Fallah,  Y. ElMadani, K. Kellay]{O. El-Fallah,  Y. ElMadani, K. Kellay}
\address{O.El-Fallah \& Y. ElMadani, Laboratoire Analyse et Applications URAC/03\\ 
Universit\'e Mohamed V Agdal-Rabat- \\  B.P. 1014 Rabat\\Morocco}
\email{elfallah@fsr.ac.ma, elmadanima@gmail.com}

\address{K. Kellay\\Universite de  Bordeaux \\ IMB\\
351 cours de la Lib\'eration\\33405 Talence \\France}
\email{kkellay@math.u-bordeaux1.fr}
\thanks{Research partially supported by "Hassan II Academy of Science and Technology" for the first and the second authors. }
\begin{document}
\maketitle
\begin{abstract} Let $\mu$ be a positive finite measure on the unit circle and $\cD (\mu)$
 the associated Dirichlet  space. The generalized Brown-Shields conjecture asserts that an 
 outer function $f \in \cD (\mu )$ is cyclic if and only if $c_\mu (Z (f))= 0$, where $c_\mu$ is the 
 capacity associated with $\cD (\mu)$ and $Z(f)$ is the zero set of $f$. In this paper we prove that 
 this conjecture is true for measures with countable support. We also give in this case 
 a  complete and explicit characterization of  invariant subspaces.
\end{abstract}
\section{Introduction}

The Dirichlet  space $\cD (\mu)$, associated with $\mu$, consists of holomorphic functions on the unit disc whose derivatives are square integrable when weighted against the Poisson integral of $\mu$.  In this paper we study cyclic vectors and invariant subspaces  of the shift  operator on $\cD(\mu)$.  The corresponding problem for the Hardy space $H^2$ was solved by Beurling in  \cite{B}: the cyclic vectors are precisely the outer functions and the invariant subspaces are generated by inner functions.  
Brown--Shields in \cite{BS} studied cyclicity in the classical Dirichlet space $\cD$. They  proved that the set of zeros  of cyclic functions in the Dirichlet space has zero logarithmic capacity
and this led them to ask whether any outer function with this property is cyclic, see also \cite{EKR2,EKR1, HS, RS3} on the study of cyclic vectors. A series of results was obtained by Richter and Richter--Sundberg in \cite{Ri,Ri1,RS1,RS2,RS3} for Dirichlet spaces $\cD(\mu)$ and especially for the description of their invariant subspaces.  More recently,  Guillot in \cite{G} obtained  a precise characterization  of cyclic vectors  for  Dirichlet spaces associated with finitely atomic measures. We refer the reader to  \cite{EKMR}   on these problems. In this work, we focus our attention in the study of the cyclic vectors   and the invariant subspaces for the shift operator acting on the Dirichlet space $\cD(\mu)$ associated with the  measures with countable support.\\

We now introduce the necessary notation. Let $H^2$  be the classical {\it Hardy space} 
of  the open unit disc $\DD$.
  If  $\mu$  is a positive Borel measure on $\TT$,  the {\it Dirichlet   space} 
  $  \cD(\mu)$  is the set of  all functions $f\in H^2$  such that
 \[
 \cD_\mu(f)=\frac{1}{\pi}\int_\DD|f'(z)|^2P_{\mu}(z)dA(z)<\infty, 
 \]
 where $dA(re^{it})=(1/\pi)rdrdt$ denotes  the normalized area measure on $\DD$ 
 and $P_{\mu}$ is the Poisson integral of $\mu$~:
\[
P_{\mu}(z)=\int_\TT\frac{1-|z|^2}{|1-\bar\zeta z|^2}d\mu(\zeta).
\]
 The space $\mathcal{D}(\mu)$ is endowed with the  norm 
\[
\|f\|_{\mu}^{2}:=\|f\|^{2}_{H^2}+\cD_{\mu}(f).
\]
 The classical Dirichlet  space  $\cD$ is precisely $\cD (m)$ 
 where $m$ denotes the normalized Lebesgue measure.  

Given $f\in \mathcal{D}(\mu)$, we denote by $[f]_{\mathcal{D}(\mu)}$ 
the smallest invariant subspace of $\mathcal{D}(\mu)$ containing $f$; namely,  
\[
[f]_{ \mathcal{D}(\mu)}:=\overline{\{pf\text{ : } p\text{  is a polynomial}\}}.
\]
We say that $f$ is {\it cyclic} for $\mathcal{D}(\mu)$ if 
$[f]_{\mathcal{D}(\mu)}=\mathcal{D}(\mu)$.   Denote by $S$ the shift operator on $\cD(\mu)$, 
that is the multiplication by $z$ on $\cD(\mu)$. 
A closed subspace $\cM$ of  $\mathcal{D}(\mu)$ is called  invariant if
 $S\cM\subset \cM$.  The lattice of all closed invariant subspaces of the 
 shift operator will be denoted by 
$\Lat(S,\cD(\mu))$. \\

In this paper we are interested in a characterization of the cyclic functions of 
$\cD (\mu)$ and in a  description of  $\Lat(S,\cD(\mu))$. 
If $d\mu(e^{it}) = 0$ then $\cD(0)=H^2$, and,  in this case, 
Beurling's theorem asserts that all closed invariant subspaces 
are given by $\Theta H^2$, where $\Theta$ is an inner function. 
As a consequence, a function $f \in H^2$ is cyclic for $H^2$ if and only if $f$ is outer. 
In order to extend Beurling's theorem to the classical Dirichlet space,  
Richter in   \cite{Ri1} was led to introduce Dirichlet type spaces. 
First, he proved that every cyclic, analytic  $2$-isometry is unitarily 
equivalent to the shift operator on some $\cD(\mu ) $. 
Recall that a bounded operator on a Hilbert space is called $2$-isometry if 
\[
{T^*}^{2}T^2-2T^*T-I=0
\]
 and it is called {\it analytic} if $\bigcap_{n\geq0}T^n\cH=\{0\}$.  This result allowed him to prove that every invariant subspace for $\cD $ 
 is of the form  $\phi \cD (|\phi|^2dm)$, where $\phi$ is an extremal function for $\cD$, that is
 \[
 \| \phi \|_\mu =1 \quad \text{ and } \quad \langle \phi, z^n\phi\rangle_{\D(\mu)}=0,  \quad n\geq 1.
 \]
 This characterization does not allow to describe the  cyclic functions for $\cD$.
 Brown and Shields showed
in \cite{BS} that if $f$ is cyclic for $\cD$ then $f$ is outer  and the zero set of its radial limit 
\[
\cZ_\TT(f)=\{\zeta\in \TT\text{ : } \lim_{r\to1-}f(r\zeta)=0\}
\]
 is of logarithmic capacity zero. Brown and Shields further conjectured 
 that the converse is also true. This  problem remains open,  
 some partial results of this conjecture can be found in  \cite{EKMR,HS, RS3}.\\

Richter and Sundberg in \cite{RS3} extended the characterization 
of invariant subspaces to all Dirichlet  spaces. Indeed, they proved that 
\[
\Lat(S,\cD(\mu)) = \big\{ \phi \cD( |\phi|^2d\mu ):  \phi  \text{ is an extremal function for }  \cD (\mu) \big\}.
\]
As before, the description of cyclic functions remains an open problem. 
To state a general Brown-Shields conjecture we will introduce a notion 
of capacity associated with these spaces.\\

 The harmonic Dirichlet  space, $\cD^h(\mu)$, associated with $\mu$ is given by
\[
\cD^h(\mu):=\left\{f\in L^2(\TT) \text{ : } \|f\|_2^2+\cD_\mu(f)<\infty\right\},
\]
where $\cD_\mu(f)=\int_\DD|\nabla P[f]|^2P_{\mu}dA$.  Note that $\cD^h(\mu)$ 
is a Dirichlet space in the sense of Beurling--Deny,  and  following \cite{BD}, 
the $c_{\mu}$-capacity of an  open subset $U\subset \TT$ is defined by
\[
c_\mu (U):=\inf\left\{\|u\|_\mu^2 \text{ : } u\in \cD^h (\mu),  \; \; u\geq 0 \text{ and } u \geq 1 \text{ a.e. on } U \right\}.
\]
 As usual we define the $c_{\mu}$-capacity of any subset $F \subset  \TT$  by 
\[
 c_{\mu}(F)= \inf \{ c_{\mu}(U):  U \; \text{ open,  }\;  F\subset U \}.
 \]
The capacity $c_\mu$ is the Choquet capacity \cite{Ca,G}  and so for every borelian subset  $E$ of $\TT$ we have
\[
c_\mu(E)=\sup\{c_\mu(K)\text{ : }K\text{ compact }, K\subset E\}.
\]
In the case $\mu = m$, it is well known that  $c_m$ is comparable to the 
logarithmic capacity. For more details see \cite[Theorem 2.5.5]{AH}.\\

We say that a  property  holds $c_\mu $-quasi-everywhere ($c_\mu $-q.e.) 
 if it holds everywhere outside a set of $c_\mu$ -capacity $0$.  
 So,  $c_\mu $-q.e. implies a.e. Note that for every function $f \in \cD (\mu )$, 
 the radial limits of $f$ exist q.e., see \cite{G}. 
 Recall also that $c_\mu$ satisfies a  weak-type inequality, namely: 
\[
c_\mu(\{\zeta\in\TT\text{ : } |f(\zeta)|\geq t  \;\;\;\;\text{$c_\mu$-q.e.}\})
\leq \frac{\|f\|^{2}_{\mu}}{t^2},\qquad f\in \cD^{h}(\mu).
\]
  
As consequence, the invariant subspace $\cM_\mu(E)$  defined by
\[
\cM_\mu(E):=\{g\in  \cD(\mu): g|E= 0  \text{  } c_\mu\text{-q.e}\}.
\]
is closed in  $\cD(\mu)$. Using these facts it is easy to verify that if 
$f$ is cyclic in $\cD (\mu)$ then $f$ is outer and $c_\mu (\cZ_\TT(f))=0$, where  
the generalized Brown-Shields conjecture claims that the converse is also true.\\

{ For $f\in \cD(\mu)$, and  $\cM\in \Lat(S, \cD(\mu))$, denote by $\cZ(f)$ the zeros set of $f$ on the disc and  set 
\[
\underline{\cZ}(f):= \{\zeta\in \overline{\DD} \text{ : } \liminf_{z\to \zeta}|f(z)|=0\}.
\]
and  
\[
\underline{\cZ}(\cM):=\bigcap_{f\in \cM}\underline{\cZ}(f).
\]
}{
Note that both sets are closed. If $f$ is an inner function, then $\underline{\cZ}(f)$ is the spectrum of $f$, and then $f$ admits analytic continuation through $\TT\setminus  \underline{\cZ}(f)$.}

 In this paper we will prove the following two results:
\begin{theor}\label{Th1}
Let $\mu$ be a positive finite measure on $\TT$.  Let $f\in  \cD(\mu)$  
be such that $\supp \mu \cap\underline{\cZ}(f)$
 is countable. The following assertions are equivalent. 
\begin{enumerate}
\item  $f$ is cyclic for    $\cD(\mu)$.
 \item $f$  is an outer function   and  $c_\mu ({\cZ_\TT}(f))=0.$
 \end{enumerate}
\end{theor}
This theorem  asserts, in particular, that the generalized Brown--Shields conjecture 
 is true if the support of $\mu$ is countable.
We also obtain, in this case, an explicit characterization 
of  invariant subspaces of $\cD (\mu)$.\\
Let  $\cM\in \Lat(S, \cD(\mu))$, we denote by $\Theta_\cM$ the greatest 
common inner divisor  of the inner parts of the non-zero functions 
of $\cM$ and 
$$\cZ_\TT(\cM)=\bigcap_{f\in \cM} \cZ_\TT(f).$$
We have the following characterization which completes  the 
Richter--Sundberg result \cite{RS3} (see Theorem \ref{RSinvar}) in the case when  $\supp \mu$ is countable. 

\begin{theor}\label{Th2}
Let $\mu$ be a positive finite measure on $\TT$ such that 
$\supp \mu$ is countable. Let  $\cM\in \Lat(S, \cD(\mu))$ 
and let $\Theta_\cM$ be  the greatest common inner divisor of
$\cM$. Then, 
\[
\cM=\Theta_\cM H^2\cap \cM_\mu(E),
\]
 where 
 $E=\{\lambda\in  \supp \mu \text{ :  }c_\mu (\{\lambda\})>0 \text{ and } \lambda \in \cZ_\TT(\cM)\}$. 
\end{theor}
 
Note that  by \cite[Corollary 5.2]{EEK}
\[
c_\mu(\{\lambda\})>0\iff  \int_{0}^{1} \frac{dr}{(1-r)P_{\mu}(r\lambda )+(1-r)^2}<+\infty.
\]

 The plan of the paper is the following. The next section gives a  background 
 on  Dirichlet  spaces. In Section 3, we give a description 
 of invariant subspaces generated by polynomials. In Sections 4
  we collect some results on closed ideals of $\cD (\mu ) \cap H^{\infty}$.
   Sections 5 and 6 are devoted to the proof of  Theorems 1 and 2.\\


\section{Background on the Dirichlet type spaces} 
In this section we recall some results from the Richter--Sundberg papers 
 \cite{RS1, RS2, RS3} about Dirichlet type spaces which will be 
 used in the proofs of our Theorems,  see also \cite{EKMR}. \\

Every function $f\in  H^2$ has non-tangential limits almost everywhere 
on the unit circle $\TT=\partial \DD$. We denote by $f(\zeta)$ the 
non-tangential limit of $f$ at $\zeta\in \TT$ if it exists. \\

 Let $\mu$ be a positive finite measure on the unit circle, 
 the associated Dirichlet  space $\cD(\mu)$  is the set of all  analytic 
 functions $f\in H^2$, such that
\[
\cD_{\mu}(f):=\int_{\TT}\cD_\xi(f)d\mu(\xi)<\infty,
\]
 where $\cD_\xi(f)$ is the {\it local Dirichlet integral} of $f$ at
  $\xi\in \TT$ given by
\[
\cD_{\xi}(f):=\int_{\TT}\frac{|f(e^{it})-f(\xi)|^2}{|e^{it}-\xi|^{2}}\frac{dt}{2\pi}.
\]
 Here $f(\lambda)$ ($\lambda=e^{it} $or $\zeta$) is the radial 
 limit of $f$ at $\lambda$,  that is $f(\lambda)=\lim_{r\to 1-}f(r\lambda)$.

Recall that a function $I$ is inner if it is a bounded holomorphic function on $\DD$
 such that $|I|=1$ a.e. on $\TT$.  
The function $O$ is called outer if it is of the form 
\[
O(z)=\exp\frac{1}{2\pi} \int_\TT\frac{\zeta+z}{\zeta-z}\log \varphi(\zeta) |d\zeta|,  \qquad z\in \DD.
\]
where $\varphi$ is a positive function such that $\log \varphi\in L^1(\TT)$. 
 Note that  $|f|=\varphi$ a.e. on $\TT$.

\begin{theo}\label{forRS} Let $f\in H^2$ and $f=BS_\nu f_o$ with $B$ a Blaschke product,  $\nu$ a singular measure associated with  $S_\nu$  a singular inner factor of $f$, $f_o$  an outer function and let $\zeta\in \TT$
 such that $f_o(\zeta)$ exists. Then 
\begin{multline}\label{formuleRS}
\cD_\zeta(f)=\sum_{\gamma\in \cZ(B)}P_{\delta_\gamma}(\zeta)|f_o(\zeta)|^2+2\int_\TT\frac{d\nu(\lambda)}{|\lambda-\zeta|^2}
|f_o(\zeta)|^2 \\
+\int_\TT\frac{|f(\lambda)|^2-|f(\zeta)|^2-2|f(\zeta)|^2
\log|f(\lambda)/f(\zeta)|}{|\lambda-\zeta|^2}\frac{|d\lambda|}{2\pi},
\end{multline}
where $\delta_\gamma$ is the dirac measure on $\gamma$. 

\end{theo}

\begin{proof} 
See \cite[Theorem 3.1]{RS1}.

\end{proof}

\begin{theo}\label{fgincRS} 
Let $f,g\in \cD(\mu)$.  If $|f(z)|\leq |g(z)|$ on $\DD$, 
then $[f]_{\cD(\mu)}\subset[g]_{\cD(\mu)}.$
\end{theo}
\begin{proof}
See \cite[Theorem 4.1]{RS1}.

\end{proof}

If $f$ and $g$ are outer functions, then we define the outer function function $f\wedge g$ by 
$\min\{|f(e^{it})|,|g(e^{it})|\}$ { and $f\vee g$ by $\max\{|f(e^{it})|,|g(e^{it})|\}$  . 

\begin{theo} \label{vee}Let $f,g$ be outer functions in $\cD(\mu)$, then $f\wedge g$ and $f\vee g$ belongs to $\cD(\mu)$ and $\|f\wedge g\|_\mu^2\leq \|f\|_\mu^2+\|g\|_\mu^2$ and  $\|f \vee g\|_\mu^2\leq \|f\|_\mu^2+\|g\|_\mu^2$.
\end{theo}
\begin{proof} see \cite[Lemma 2.2]{RS2}
\end{proof}
}

\begin{theo}\label{fginterRS}
Let $f,g\in \cD(\mu)$ be outer functions and let $h$ be the outer function 
given by $|h|:=|f|\wedge |g|$ a.e on $\TT$. Then 
$[h]_{\cD(\mu)}=[f]_{\cD(\mu)}\cap[g]_{\cD(\mu)}$. If further  $fg\in \cD(\mu)$, then
 $$[fg]_{\cD(\mu)}=[f]_{\cD(\mu)}\cap[g]_{\cD(\mu)}.$$
\end{theo}

\begin{proof}
See \cite[Theorem 4.5]{RS2}.

\end{proof}

\begin{theo}\label{RSthm}
Let $\cM\in\Lat(S,\cD(\mu))$. Then there exists a multiplier $\phi$ of $\cD(\mu)$ 
such that $\cM=[\phi]_{\cD(\mu)}$.
\end{theo}

\begin{proof} 
See \cite[Theorem 3.5]{RS2} and \cite[Theorem 3.2]{RS3}.

\end{proof}
 
\begin{theo}\label{RSinvar} Let  $\cM\in \Lat(S,\cD(\mu))$, and let 
$\Theta_\cM$ the greatest common inner divisor  of  $\cM$. 
Then, there is an  outer  $f\in   \cD(\mu)\cap H^{\infty}$ such that 
\[
\cM=\Theta_\cM H^2\cap [f]_{  \cD(\mu)}. 
\]
In fact $f$ can be chosen so that $f$ and $\Theta_\cM f$ 
are multipliers of $\cD(\mu)$. 
\end{theo}

\begin{proof}
See \cite[Theorem 5.3]{RS2}.

\end{proof}


\section{Invariant subspaces generated by polynomials }

In this section we characterize closed invariant subspaces generated by polynomials.\\
Let $\zeta \in \TT$. We will say that $\zeta$ is a bounded point evaluation 
of $\cD (\mu)$ if there exists a constant $M>0$ such that for every polynomial $p$ 
$$
|p(\zeta)|\leq M \|p\| _\mu.
$$

It means that the functional $p\to p(\zeta)$ extends to a continuous functional on 
$\cD (\mu)$. Since the polynomials are dense in $\cD (\mu)$,  this extension is unique and will be denoted by $L_\zeta(f)=f(\zeta)$.\\

{Let
 \[
c_\mu ^a(E)=\inf  \big\{ \| f\| _\mu ^2 :\ f \in \cD (\mu)\ \mbox{and }\ |f|\geq 1 \text{ a.e. on a neighborhood of }\ E\big\}.
\]
Since 
\[
\| f\| _\mu \geq \| f_o \| _\mu \geq \| f_o \wedge 1\| _\mu ,
\]
 where $f_o$ is the outer part of $f$,  { see for instance \cite[Corollary 7.6.2]{EKMR}},  then
\[
c_\mu ^a(E)=\inf  \big\{ \| f\| _\mu ^2 :\ f \in \cD (\mu) \text{ is an outer function,  } |f|= 1
 \text{ a.e. on a neighborhood of } E\big\}.
\]
Note that $c_\mu(E)\leq c_\mu^a(E)$. In fact  $c_\mu^a$ is comparable to $c_\mu$ by the following Lemma { (see \cite[Theorem 38]{G1})}. For the sake of completeness,  we give the proof
\begin{lem} We have $c_\mu(E)\leq c_\mu^a(E)\leq 4c_\mu(E)$
\end{lem}
\begin{proof}
Let $\varepsilon\geq 0$ and let $f\in\cD^h(\mu)$ such that $|f|\geq 1$ a.e on a neighborhood of $E$ and $\|f\|_\mu^2\leq c_\mu(E)+\varepsilon$. Write $f=f_1+\overline{f_2}$ where $f_i\in \cD(\mu)$, $f_2(0)=0$, $\|f\|_\mu^2=\|f_1\|_\mu^2+\|f_2\|_\mu^2$ and so for $i=1$ or $i=2$, we have  $|f_i|\geq 1/2$ a.e on a neighborhood of $E$.  

Consider the outer function 
$\varphi=f_1 \vee f_2$ a.e. on $\TT$.  By Theorem \ref{vee},  $\varphi\in \cD(\mu)$ and 
$\|\varphi\|_\mu^2\leq \|f_1\|_\mu^2+\|f_2\|_\mu^2 =\|f\|_\mu^2 $. Since $|\varphi|\geq 1/2$ a.e on a neighborhood of $E$, we get 
$c_\mu^a(E)\leq 4\|\varphi\|_\mu^2\leq 4|f\|_\mu^2\leq 4c_\mu(E)+\varepsilon$.
\end{proof}
}
Using this fact, we have the following lemma.

\begin{lem}\label{capacityevaluation}
Let $\zeta \in \TT$. The following properties are equivalent: 
\begin{enumerate}
\item $\zeta$ is a bounded point evaluation of $\cD (\mu)$.
\item The polynomial $z-\zeta$ is not cyclic for $\cD (\mu)$.
\item  $c_{\mu}(\zeta) >0$.
\end{enumerate}
\end{lem}

\begin{proof}   
$(1) \Longrightarrow (3)$: Suppose that  $\zeta$ is a bounded point evaluation for 
$\cD (\mu)$. We will prove that  $c_{\mu} ^a(\zeta) >0$. Let $f \in \cD (\mu)$  
be an outer function such that $|f|=1$ a.e. on an open arc $I$ centered at $\zeta$. 
Then, there is $\lambda \in \TT$ such that
\[
f(z) =\lambda \exp \frac{1}{2\pi}\displaystyle \int _{\TT \setminus I} \frac{\eta +z}{\eta -z}\log |f(\eta)|d\eta|
 , \qquad z\in \DD.
\] 
Hence  $f$ is analytic in a  neighborhood of $\zeta$. Since $L_\zeta $ is bounded, 
we have $|L_\zeta (f)|= |f(\zeta)|= 1$. So  $\| f\| _\mu  \geq 1/\| L_\zeta \| $, 
and  consequently, $c_{\mu} ^a(\zeta) \geq 1/\|L_\zeta \|^2$. This proves that  (1) implies (3). \\

$(3) \Longrightarrow (2)$: Suppose that  $c_{\mu}(\zeta) >0$. Then $\cM_\mu (\{\zeta\} )$ is a proper closed invariant subspace of $\cD (\mu )$ and $z-\zeta$ is not cyclic. \\

$(2) \Longrightarrow (1)$: Suppose now that $z-\zeta$ is not cyclic for $\cD (\mu)$. 
The space $\cD (\mu)/[(z-\zeta)]_{\cD(\mu)}$ is of dimension one. 
Let $\pi $ be the canonical surjection defined by
\[
\begin{array}{lllll}
\pi :&\cD (\mu )& \to  & \cD (\mu)/[(z-\zeta)]_{\cD(\mu)}\\

&f &\to & \pi (f) = f+ [(z-\zeta)]_{\cD(\mu)}\\
\end{array}
\]
and let $\rho$ be the isomorphism from  $\cD (\mu)/[(z-\zeta)]_{\cD(\mu)}$  
to $\CC$ satisfying $\rho (\pi (1))=1$. It is easy to verify that 
$L_\zeta = \rho \circ \pi$, which proves that  $L_\zeta $ is continuous.

\end{proof}

\begin{lem}\label{cyclicpol}
Let $\zeta \in \TT$ and let $n$ be an integer $n\geq 1$. We have
\begin{enumerate}
\item If $c_{\mu}(\zeta) =0$ then $[(z-\zeta)^n]_{\cD(\mu)}= \cD(\mu)$.
\item  If $c_{\mu}(\zeta) >0$ then $[(z-\zeta)^n]_{\cD (\mu)}= \cM _\mu({\zeta})$.
\end{enumerate}
\end{lem}

\begin{proof} 
By Theorem \ref {fginterRS}, we have 
$[(z-\zeta)^n]_{\cD(\mu)}= [(z-\zeta)]_{\cD(\mu)}$. 
The two assertions come from Lemma \ref{capacityevaluation}.

\end{proof}

To state the characterization of closed invariant subspaces generated 
by polynomials, we need some notations. \\
To any  $\Lambda =\{ (z_1, n_1),(z_2,n_2),\ldots, (z_k,n_k)\}$,
 where $z_j \in \DD$ and  $n_j\in  \NN ^*$,  we associate  a polynomial 
 $p_ \Lambda  =\displaystyle \Pi _{j=1}^k(z-z_j)^{n_j}$. Let  $E\subset \TT$, we define
\[
{\cM }_\mu ( \Lambda,E)= \{f\in \cD (\mu ):\ f \in p_\Lambda H^{2}\ \mbox{and }\ f_{|E}=0\}.
\]
For a polynomial $p$, let $\Lambda _p= \{ (z,n)\in \DD \times \NN ^*: z  \text{ is a zero of $p$ of order }  n\} $.  
We have 

\begin{theo}\label{pol} Let $p$  be a polynomial. Then
\[
[p]_{\cD (\mu)}= {\cM }_\mu ( \Lambda _p,E),
\]
where $E= \{\zeta \in \TT: \ p(\zeta)=0\ \mbox{and}\ c_{\mu}(\zeta) >0\}$.
\end{theo}

\begin{proof}
The proof is based on classical arguments and on the  Lemma \ref {cyclicpol}.

\end{proof}


\section{Closed ideals of $\cD (\mu) \cap H^{\infty}$ }

First, we will state the following lemmas which will be used in the sequel.
\begin{lem}\label{fglem}
Let $f\in  \cD(\mu)\cap H^{\infty}$ and $g\in H^{\infty}$ be outer functions 
such that  $\|f\|_{\infty}\leq 1$ and $\|g\|_{\infty}\leq 1$.  Let $h$ be the outer function given by 
\[
|h(\zeta)|= \left\{
  \begin{array}{ll}
    |f(\zeta)|, & \text{ a.e. on }\; V,\\
    |g(\zeta)|, & \text{ a.e. on }\;  \TT\setminus V, 
  \end{array}
\right.
\]
where  $V$ is a closed neighborhood
 of $\supp(\mu)$.
Then $h\in  \cD(\mu)$  and 
\begin{equation}\label{eqh}
\cD_{\mu}(h)\leq \cD_{\mu}(f)+\frac{\mu(\TT)}{\dist(\mathbb{T}\setminus V, \supp(\mu))^2}
\Big(\|g\|^2_2+2\| \log1/|g|\|_{L^{1}(\TT)}\Big).
\end{equation}
\end{lem}
\begin{proof}
Note that  $
\|h\|^{2}_2\leq \|f\|^2_2+\|g\|^2_2$, so $h\in H^2.$  {Since $f\in \cD(\mu)$, $\cD_\zeta(f) <\infty$ for  $\mu$--almost every $\zeta\in \TT$.} 
By  \eqref{formuleRS}, for $\zeta\in \supp \mu$,  let   
$\delta=\dist(\mathbb{T}\setminus V, \supp(\mu))$, we have
\begin{eqnarray*}
\cD_{\zeta}(h)&
= &\displaystyle \int_{\lambda\in V}+\displaystyle 
\int_{\lambda\in \mathbb{T}\setminus V}
\frac{|h(\lambda)|^2-|h(\zeta)|^2-2|h(\zeta)|^2\log|{h(\lambda)}/{h(\zeta)}|}{|\lambda-\zeta|^2}\frac{|d\lambda|}{2\pi}\nonumber \\
&\leq  &\cD_{\zeta}(f)+\frac{1}{\delta^2}\displaystyle \int_{\lambda\in \mathbb{T}\setminus V}\Big(|g(\lambda)|^2+2|f(\zeta)|^2
\log\frac{1}{|g(\lambda)|}-|f(\zeta)|^2\log\frac{e}{|f(\zeta)|^2}\Big)\frac{|d\lambda|}{2\pi}\nonumber \\
& \leq & \cD_{\zeta}(f)+\frac{1}{\delta^2}\Big(||g||^2_2+2|f(\zeta)|^2 \displaystyle
 \int_{\mathbb{T}}\log \frac{1}{|g(\lambda)|}\frac{|d\lambda|}{2\pi}\Big).  \nonumber 
\end{eqnarray*}

It's clear that \eqref{eqh} follows from this inequality and thus $h\in \cD(\mu)$.

\end{proof}

\begin{lem}\label{lemmefusion}Let $V=(e^{ia},e^{ib})$,  and let $f\in \cD(\mu)\cap H^\infty$ 
be an outer function.  Let  $f_{V}$ be the outer function defined by 
\[
|f_{V}(\zeta)|=\left\{
\begin{array}{lll}
 |(\zeta-e^{ia})(e^{ib}-\zeta)| |f(\zeta)|,& \text{ a.e. on \; } {\overline{V}},\\
 |(\zeta-e^{ia})(e^{ib}-\zeta)| ,& \text{ on \; }\TT\setminus {\overline{V}}.
\end{array}
\right.
\]
Then $f_{V}\in \cD(\mu)\cap H^\infty$.
\end{lem}

\begin{proof} 
Set $u(z)=(z-e^{ia})(e^{ib}-z)$, again by  \eqref{formuleRS}, for $\zeta\in \supp \mu$, we have
\begin{eqnarray*}
\cD_{\zeta}(f_V)
&=&\int_{\lambda\in V}+\int_{\lambda\in \mathbb{T}\setminus V}\frac{|u(\lambda)|^2-|u(\zeta)|^2-2|u(\zeta)|^2
\log|{u(\lambda)}/{u(\zeta)}|}{|\lambda-\zeta|^2}\frac{|d\lambda|}{2\pi}\nonumber \\
&\leq& \cD_{\zeta}(uf)+ D_\zeta(u).
\end{eqnarray*}
Clearly,  $f_V\in \cD(\mu)\cap H^\infty$.

\end{proof}

Recall that   $\cD(\mu)\cap H^{\infty}$ is a Banach algebra endowed  with 
the pointwise multiplication and equipped with the norm 
\[
\|f\|_{\infty,\mu}=\|f\|_\infty+\cD_\mu(f)^{1/2}.
\]

\begin{theo}\label{th4}
Let  $\cJ$ be a closed ideal of $\cD(\mu)\cap H^{\infty}$. 
Let  
\[
\pi:  \cD(\mu)\cap H^\infty\longrightarrow  \cD(\mu)\cap H^\infty/\cJ,
\]
 be the canonical surjection. Then 
\[
\sigma(\pi(u))= \underline{\cZ}(\cJ)
\]
 where  $u:z\to z$ is the identity map and $\sigma(\pi(u))$ is the spectrum of  $\pi(u)$.
\end{theo}

{ \subsection*{Remark} Note that $\sigma(\pi(u))=\emptyset$  if and only if $ \cJ=\cD(\mu)\cap H^\infty$, see  Lemma \ref{lem6}  for more general statement.}

\subsection*{Proof Theorem \ref{th4}}
First, we prove that $ \underline{\cZ}(\cJ)\subset \sigma(\pi(u))$. 
Let $\lambda \notin \sigma(\pi(u))$,  then there exists $f \in \cD(\mu)\cap H^{\infty}$ and $g \in \cJ$ such that $(\lambda -z)f(z) =1-g(z)$. Which gives obviously that $\displaystyle \lim _{z\to \lambda}g(z)=1$ and $\lambda \notin \underline{\cZ}(\cJ)$.\\

Conversely,  let  $\lambda\notin \underline{\cZ}(\cJ)$. {We suppose that $|\lambda|=1$ since the case $|\lambda|\neq 1$ is obvious}. Hence there exists $g\in  \cJ$  and $c>0$ 
such that $ |g(z)|\geq c$  on  a neighborhood  of $\lambda$,  
$V_\lambda=\{z\in \DD\text{ : } |z-\lambda|\leq \delta/2\} $. 

Note that if  $w\in \cC^2(\TT)$ and  $\log |w|\in L^1(\TT)$ then the outer function 
$f_w$ defined by $|f_w|=w$,  a.e.  on $\TT$,  
belongs to $\cC^2(\overline{\DD})\cap\Hol(\DD)\subset \cD(\mu)$.    
Now consider the smooth function $w\in \cC^2(\TT)$ such that $w\geq c>0$  and 
\[
w(\zeta)=\left\{\begin{array}{lll}
c&\text{ if }&|\zeta-\lambda|\leq \delta/2,\\
1&\text{ if }&|\zeta-\lambda|\geq  \delta.\\ 
\end{array}
\right.
\]
Suppose that $\|g\|_\infty\leq 1$.    Write $g=g_ig_o$ the inner--outer factorization of $g$. 
Consider the outer function  
$$|h|=|g_o|\wedge |f_w| =|g_o|\wedge |w|\qquad \text{ a.e. on \; } \TT.$$   
{Note that $|h|= c$ near $\lambda$ and  by Theorem \ref{fginterRS},   $0\neq h\in\cD(\mu)$.}  
Now consider the outer function 
$$|\widetilde{h}(e^{it})|=(|g_o|\vee |f_w|)(e^{it})=\max\{|g_o(e^{it})|,w(e^{it})\},\qquad \text{ a.e. on \; }\TT$$
Clearly,  
\[
gf_w=  g_i(g_o\wedge f_w )\times (g_o\vee f_w)= g_i h \times \widetilde{h} \quad \text{ on }\; \DD
\]
and $ \widetilde{h} \in\cD(\mu)\cap H^\infty$ is invertible. So $g_ih\in \cJ$.  { Since $|g(z)|\geq c$ on $V_\lambda$, then $\lambda$ doesn't  belong to the spectrum  of the inner function $g_i$.  
Therefore  $g_i$ is analytic across an arc that contains $\lambda$, so $g_i(\lambda)$ is well-defined and $|g_i(\lambda)|=1$}. 
Let  
$$\psi=\frac{1}{(g_ih)(\lambda)}\frac{g_ih-g_ih(\lambda)}{u-\lambda}.$$ 
{ Since $(g_i h)(\lambda)\neq 0$, the function $\psi$ is well-defined}.  Note that  the function  $|g_o|\wedge |f_w|=|f_w| $ on 
$\{\zeta\in \TT\text{ : } |\zeta-\lambda|\leq \delta/2\}$  is $\cC^\infty$ on $V_\lambda$ and 
\[
h(z)=\exp\int_{|\zeta-\lambda|\leq\delta/2} \frac{\zeta+z}{\zeta-z}\log c\frac{|d\zeta|}{2\pi} \times
 \exp\int_{|\zeta-\lambda|\geq \delta/2} \frac{\zeta+z}{\zeta-z}\log |h(\zeta)|\frac{|d\zeta|}{2\pi},
 \]
 the two functions are clearly $\cC^\infty$ on $V_\lambda$ and 
 hence $\psi$ is $\cC^2$ on $V_\lambda$. Thus  $\psi\in \cD(\mu)\cap H^\infty$. So 
   $\pi(\psi)(\lambda-\pi(u))=\pi(1)$  and   we get 
$\lambda\notin \sigma(\pi(u))$ which finishes the proof.

\begin{lem}\label{lem68} 
Let  $\cJ$ be a closed ideal of $\cD(\mu)\cap H^{\infty}$ such that $\cJ$ 
contains  an outer function,   and  $\underline{Z}(\cJ)=\{\zeta_0\}\subset \TT$. 
Then $(z-\zeta_0)\in \cJ$.
\end{lem}

\begin{proof} 
We suppose that $\zeta_0=1$. Consider again the canonical surjection
\[
\pi:  \cD(\mu)\cap H^\infty\longrightarrow  \cD(\mu)\cap H^\infty/ \cJ,
\]
and let $u:z\to z$ be the identity map. Given $\lambda\in \DD$  and $f\in \cJ$, we define
\[
L_\lambda(f)(z)=\left\{
\begin{array}{lll}
\displaystyle \frac{f(z)-f(\lambda)}{z-\lambda}\;,& z\in \DD\setminus\{\lambda\},\\
f'(\lambda)\;,&z=\lambda.
\end{array}
\right.
\]
Since $L_\lambda f-L_0 f=\lambda L_\lambda L_0 f$, it is obvious that 
$L_\lambda(f)\in \cD(\mu)\cap H^\infty$.  So  
\[
\pi(L_\lambda(f))(\lambda\pi(1)-\pi(u))=f(\lambda)\pi(1),\qquad \lambda\in  \DD.
\] 
{The operator $T$ defined on $\cD(\mu)\cap H^\infty/\cJ$ by
$$\pi(f)\to \pi(u)\pi(f),$$
has the  spectrum $\sigma(T)=\{1\}$ and for $n\geq 0$
\begin{equation}
\label{pol}
\|T^n\|=\|\pi(u)^n\|\leq \|u^n\|_\mu=\sqrt{1+n\mu(\TT)}.
\end{equation}
On the other hand,  let $f\in \cJ$ be an outer function. So  
\[
 (\lambda\pi(1)-\pi(u))^{-1}=( {1}/{f(\lambda)})\pi(L_\lambda(f)),\qquad \lambda\in \DD.
 \]
Since $f\in H^\infty $ is outer,  for all $\varepsilon >0$, $1/|f(\lambda)|=O(e^{\varepsilon/(1-|\lambda|})$ as $|\lambda|\to1-$. Therefore,  
\[
\|(\lambda I -T)^{-1}\|=O\Big( \exp\frac{\varepsilon}{1-|\lambda|}\Big), \qquad |\lambda|\to 1-,
\]
where $I$ is the identity operator. On the other hand,  by Cauchy inequality, see \cite[Lemma 2]{At}, for all $\varepsilon>0$
\begin{equation}
\label{exp}
\|T^{-n}\|=O(e^{\varepsilon \sqrt{n}}),\qquad n\to +\infty.
\end{equation}
Therefore  the operator $T$ is invertible, $\sigma(T)=\{1\}$ satisfies \eqref{pol} and \eqref{exp}, then it follows from Phragm\'en Lindel\"of principle \cite[Corollary 1]{At}, that  $(I-T)^2=0$, which means that 
$$(1-z)^2\in \cJ.$$} 

Let a functional $\ell$ in the dual space of $\cD(\mu) \cap H^{\infty}$ be  such that  $\ell$ is orthogonal to $\cJ$. We have  
\[
\langle \ell , z^n(1-z)^2 \rangle =0,\qquad n\geq 0.
\]
Let $\widehat{\ell}(n)= \langle \ell, z^n\rangle$, the last equality implies that
\[
\widehat{\ell}(n)-2 \widehat{\ell}(n+1)+\widehat{\ell}(n+2)=0\qquad n\geq 0.
\]
 So $\widehat{\ell}(n)=(a+bn)$ for some constants $a,b\in \CC$. In the other hand 
\[
| \widehat{\ell}(n)|\leq
  \|\ell\|\;  \|z^{n}\|_{\cD(\mu) \cap H^{\infty}}=  O (\sqrt{n}).
  \] 
  Hence $b=0$ and $\langle \ell , (1-z) \rangle =0.$
  Which gives $1-z \in \cJ$ and completes the proof.
   \end{proof}

\subsection*{Remark} 
The preceding result can be extended to closed ideals  
such that  their greatest common inner divisor is $1$. 


\section{Cyclicity in $\cD (\mu)$ }

\begin{lem}\label{lem6}
Let $f\in  \cD(\mu)\cap H^\infty$ be an outer function. If 
$\underline{\cZ}(f)\cap \supp(\mu)=\emptyset$ then  $f$ is cyclic for  $ \cD(\mu)$.
\end{lem}

\begin{proof}
Suppose that $\|f\|_\infty\leq 1$. Since  $\underline{\cZ}(f)\cap \supp \mu=\emptyset$, 
there exists a  closed neighborhood $V$
 of $\supp \mu$ such that   $\underline{\cZ}(f)\cap V=\emptyset$.   
 Consider the following outer functions
\[
|h(\zeta)|= \left\{
  \begin{array}{lll}
    1, & \text{ on \; } &V,\\
    |f(\zeta)|, & \text{ a.e. on\;} & \TT\setminus V,
  \end{array}
\right.
\]
and 
\[
|g(\zeta)|= \left\{
  \begin{array}{lll}
    |f(\zeta)|, & \text{ a.e. on \; }&V,\\
    1, &\text{ on \; }& \TT\setminus V.
  \end{array}
\right.
\]
By  Lemma \ref{fglem}, $h,g\in \cD(\mu)\cap H^\infty$  and by Theorem \ref{fginterRS}
\[
[f]_{ \cD(\mu)}=[hg]_{ \cD(\mu)}  =[h]_{ \cD(\mu)}\cap [g]_{ \cD(\mu)}.
\]

Note that   $\underline{\cZ}(g)=\underline{\cZ}(f)\cap V=\emptyset$.  
Hence $|g(z)|\geq c>0$ on $\DD$,  and then  { by Theorem \ref{fgincRS}, $g$ is cyclic for $\cD(\mu)$}.  So 
\begin{equation}
\label{hf}
[f]_{ \cD(\mu)}=[h]_{ \cD(\mu)}.
\end{equation}
Let  $\delta>0$, and consider the following outer function 
\[
|h_{\delta}(\zeta)|= \left\{
  \begin{array}{lll}
    1, &  \text{ on  \;}& V,\\
  1/(|f(\zeta)|+\delta), & \text{ a.e. on \;  } &\TT\setminus V
  \end{array}
\right.
\]
We have  $\lim_{\delta\rightarrow 0}|h_{\delta}h(\zeta)|=1\text{ a.e on } \TT$. 
Since $\underline{\cZ}(f)\cap \supp(\mu)=\emptyset$,  by Lemma \ref{fglem}

\begin{eqnarray*}
\cD_{\mu}(h_{\delta}h(\zeta))&\leq& \frac{1}{\dist(\supp \mu,\mathbb{T}\setminus V)^2}
\Big(\mu(\mathbb{T})+2\Big\|\log\frac{|f|+\delta}{|f|}\Big\|_{L^{1}(\TT)}\Big) \nonumber \\
&\leq &\frac{1}{\dist(\supp \mu,\mathbb{T}\setminus V)^2}\Big(\mu(\mathbb{T})+4\|\,|\log|f||\,\|_{L^{1}(\TT)}\Big).\nonumber
\end{eqnarray*}
So
\[
\liminf_{\delta\to 0}D_{\mu}(h_{\delta}h)<\infty.
\]
Since  $h_\delta$ and $h$ are outer, $|h_{\delta}h(z)|\leq |h(z)|/\delta$ for $z\in \DD$, 
by  Theorem \ref{fgincRS}, we have $h_{\delta}h\in [h]_{ \cD(\mu)\cap H^\infty}$ for all $\delta>0$.   
So  $h$ is cyclic  for $ \cD(\mu)$ and by \eqref{hf} $f$ is also cyclic  for $ \cD(\mu)$.

\end{proof}

\begin{lem}\label{th7}
Let $f\in  \cD(\mu)\cap H^\infty$ be  outer function. Then 
\[
\underline{\cZ}([f]_{ \cD(\mu)})\subset \underline{\cZ}(f)\cap\supp \mu.
\]
\end{lem}

\begin{proof}  Let  $\zeta \in \underline{\cZ}(f)\setminus \supp(\mu)$, there exists  an open  neighborhood $V_\zeta$ of $\zeta$ such that
{$\overline{V_{\zeta}}\cap \supp(\mu)=\emptyset.$ } Write 
\[
f=f_1\times f_2,   \qquad f_1,f_2\in  \cD(\mu)\cap H^{\infty},
\]
where $\underline{\cZ}(f_2)\subset \mathbb{T}\setminus V_{\zeta}$ and  {$\underline{\cZ}(f_1)\subset \overline{V_{\zeta}}$}, so  $\underline{\cZ}(f_1)\cap \supp(\mu)=\emptyset.$  By Proposition \ref{lem6},  $f_1$ is cyclic for  $ \cD(\mu)$.  So 
\[
[f]_{ \cD(\mu)}=[f_1f_2]_{ \cD(\mu)} =[f_1]_{ \cD(\mu)}\cap [f_2]_{ \cD(\mu)} =[f_2]_{ \cD(\mu)}.
\]
Hence  $\zeta\notin \underline{\cZ}([f]_{ \cD(\mu)})$.
\end{proof}

{
\begin{lem}\label{casinfty}
Let $f\in \cD(\mu)$ be an outer function. Then 
\begin{enumerate}
\item $\underline{\cZ}(f)=\underline{\cZ}(f\wedge 1)$
\item $[f]_{\cD(\mu)}=[f\wedge 1]_{\cD(\mu)}$
\item $\underline{\cZ}([f]_{\cD(\mu)})=\underline{\cZ}([f]_{\cD(\mu)}\cap H^\infty)$
\end{enumerate}
\end{lem}
\begin{proof} \textit{(1)} is obvious. \textit{(2)}  is a particular case of  Theorem \ref{fginterRS}. 
\textit{(3)} follows from \textit{(1)} and \textit{(2)}.
\end{proof}}
\subsection*{Proof Theorem \ref{Th1}}

If  $f$ is cyclic for  $ \cD(\mu)$, then it's clear that  $f$ is an outer function and $c_\mu ({\cZ}_\TT(f))=\emptyset$. \\

Conversely, let  $f\in \cD(\mu)$ be outer function, since $c_\mu ({\cZ_\TT}(f\wedge 1))= c_\mu ({\cZ_\TT}(f)) $, { by Lemma \ref{casinfty}, we can suppose that $f\in H^\infty$.}

Now suppose   $c_\mu ({\cZ_\TT}(f))=0$, then 
by Lemma  \ref{th7} we have 
$$ \underline{\cZ}([f]_{ \cD(\mu)})\subset \supp \mu\cap \underline{\cZ}(f).$$

If $\underline{\cZ}([f]_{ \cD(\mu)})=\emptyset.$ { By Lemma \ref{casinfty},  $\underline{\cZ}([f]_{\cD(\mu)}\cap H^\infty)=\emptyset$ and by Theorem \ref{th4},  $\sigma(\pi(u))=\emptyset$ where $u:z\to z$ is identity map and $\pi$ the canonical projection associated with $[f]_{\cD(\mu)}\cap H^\infty$. 
Then $f$ is cyclic  and the Theorem is proved.}

Suppose now that  
\[
\underline{\cZ}([f]_{ \cD(\mu)})\neq\emptyset.
\]
We will show that this leads to a contradiction. Since $ \underline{\cZ}([f]_{ \cD(\mu)})$ is a countable set,   $ \underline{\cZ}([f]_{ \cD(\mu)})$  
have an isolated  point, noted by  $\zeta_0$. Let $V=(e^{ia},e^{ib})$  be a  neighborhood 
of $\zeta_0$ such that ${\overline{V}}\cap  \underline{\cZ}([f]_{ \cD(\mu)})=\{\zeta_0\}$.   
Consider the outer functions  $h=f_V$ and $g=f_{\TT\setminus V}$ : 

\[
|h(\zeta)|= \left\{
  \begin{array}{ll}
   |(\zeta-e^{ia})(e^{ib}-\zeta)| |f(\zeta)|, & \text{ a.e. on \; }\overline{V},\\
   |(\zeta-e^{ia})(e^{ib}-\zeta)|, & \text{ on } \TT\setminus \overline{V},
  \end{array}
\right.
\]
and
\[
|g(\zeta)|= \left\{
  \begin{array}{ll}
     |(\zeta-e^{ia})(e^{ib}-\zeta)|, & \text{ on \; } \overline{V},\\
  |(\zeta-e^{ia})(e^{ib}-\zeta)|  |f(\zeta)|, & \text{ a.e. on \;} \TT\setminus \overline{V}.
  \end{array}
\right.
\]

By Lemma \ref{lemmefusion},   $h,g\in \cD(\mu)\cap H^\infty$ and $hg \in [f]_{ \cD(\mu)}.$ Define the closed division ideal $\cJ$ by
 \[
 \cJ:=\{\psi\in  \cD(\mu) \cap H^\infty \text{ : }\psi g\in [f]_{ \cD(\mu)}\}.
 \] 

Since $[f]_{\cD(\mu)}\cap H^\infty\subset \cJ$ and $h\in \cJ$,
\[
\underline{\cZ}(\cJ)\subset \underline{\cZ}([f]_{ \cD(\mu)})\cap \underline{\cZ}(h)\subset  \{\zeta_0\}.
\]
Hence  by Lemma \ref{lem68},   $z-\zeta_0\in \cJ$ (In fact, if $\underline{\cZ}(\cJ)=\emptyset$ then  $\cJ=\cD_\mu$) and then $(z-\zeta_0)g\in [f]_{\cD(\mu)}$. 
Since by Theorem  \ref{fginterRS}, we have $[(z-\zeta_0)g]_{\cD(\mu)}=
[(z-\zeta_0)]_{\cD(\mu)}\cap [g]_{\cD(\mu)}$. We get 
\[
[(z-\zeta_0)]_{\cD(\mu)}\cap [g]_{\cD(\mu)}\subset [f]_{\cD(\mu)}.
\]

Now we distinguish  2 cases.

\begin{itemize} 
\item  If  $c_\mu (\{\zeta_0\})=0$, then  $z-\zeta_0$ is cyclic and hence 
$g\in [f]_{\cD(\mu)}$.

\item  If  $c_\mu (\{\zeta_0\})>0$,   since $c_\mu ({\cZ_\TT}(f))=0$, then $\zeta_0\notin \cZ_\TT(f)$ 
and by Lemma \ref{capacityevaluation},  $h(\zeta_0)$ exists and $h(\zeta_0)\neq 0$.  { 
Note also that $g(\zeta_0)$ exists, indeed $|g(\zeta)|=|f_{\TT\backslash V}(\zeta) |=|(\zeta-e^{ia})(\zeta-e^{ib})|$ on $V$, so  the outer function  $g$ can be continued holomorphically  across $V$, see \cite[p. 65]{EKMR}.} Now write 
$$g=\underbrace{\frac{h}{h(\zeta_0)}g}_{\in [f]_{\cD(\mu)}}+
\underbrace{\frac{h(\zeta_0)-h}{h(\zeta_0)}g}_{\in [f]_{\cD(\mu)}} \in [f]_{\cD(\mu)}.$$

\end{itemize}
In the two cases   $g\in [f]_{\cD(\mu)}$ which gives a contradiction 
since and $\zeta_0\in \underline{\cZ}([f]_{\cD(\mu)})$ and $g(\zeta_0)\neq 0$.  
Thus we have completed the proof of the first main result.


\section{Invariant subspaces of $\cD (\mu)$}

This section is devoted to the proof of Theorem 2.
Let $\cM\in \textrm{Lat}(S,\cD(\mu))$  be a closed invariant subspace  of $\cD(\mu)$. By  Theorem \ref{RSinvar}, there exists an inner function  $\Theta_\cM$ and 
an outer function  $f\in   \cD(\mu)\cap H^{\infty}$ such that 
$$\cM=\Theta_\cM H^2\cap [f]_{  \cD(\mu)}. $$ 
We need to show that 
\[
[f]_{  \cD(\mu)}= \cM_\mu(E), 
\]
here  $\cM_\mu(E)$ be the closed invariant subspace of $  \cD(\mu)$ given by
\[
\cM_\mu(E)=\{\psi\in \cD(\mu)\text{ : } \psi|E=0\},
\]
where $E=\{\lambda\in  \supp \mu \text{ :  }c_\mu (\{\lambda\})>0 \text{ and  }\lambda\in \cZ_\TT(\cM)\}.$\\

 Note that if $\lambda \in \supp \mu$  such that $c_\mu (\{\lambda\})>0$ then 
 $$\lambda\in  \cZ_\TT(\cM) \iff \lambda \in \cZ_\TT(f).$$

Let $\psi \in [f]_{  \cD(\mu)}$,  and let $\lambda\in E$, since $c_\mu(\{\lambda\})>0$,  { the evaluation is continuous and by Lemma \ref{capacityevaluation}},  
$\psi(\lambda)$ exists and $\psi(\lambda)=0$.  So $\psi\in\mathcal{M}_{\mu}(E)$
 and we have  $[f]_{  \cD(\mu)}\subset \mathcal{M}_{\mu}(E)$.  \\

Now we show the opposite inclusion.  By  Theorem \ref{RSthm}, 
there exists an outer function $\phi\in\mathcal{M}_{\mu}(E)\cap H^{\infty}$ such that  
 $\mathcal{M}_{\mu}(E)=[\phi]_{  \cD(\mu)}.$  
Consider  the closed division ideal  
\[
\cJ =\{\psi\in  \cD(\mu)\cap H^\infty: \psi\phi\in [f]_{  \cD(\mu)}\}.
\]
We have $ [f]_{  \cD(\mu)}\cap H^\infty\subset  \cJ$ and by Lemmas \ref{th7} and \ref{casinfty} (3)
\[
\underline{\cZ}(\cJ)\subset \underline{\cZ}(f)\cap \supp \mu.
\]

\noindent $\bullet$ If $\underline{\cZ}(\cJ)=\emptyset$, then by Theorem \ref{th4}, we get  $\cJ=\cD(\mu)\cap H^\infty$. \\

\noindent $\bullet$  If  $\underline{\cZ}(\cJ)\neq \emptyset$. Since $\supp \mu$ is countable,  
the set  $\underline{\cZ}(\cJ)$ has an isolated  point  $\zeta_0$. Let $V=(e^{ia},e^{ib})$ 
 be a  neighborhood of $\zeta_0$ such that ${\overline{V}}\cap \underline{\cZ}(\cJ)=\{\zeta_0\}$.  Consider again the outer functions   $h=f_V$ and $g=f_{\TT\setminus V}$ as in Lemma \ref{lemmefusion}.  By Lemma \ref{lemmefusion},   
$g,h\in \cD(\mu)\cap H^\infty$.
 Now let 
\[
\widetilde{\cJ}:=\{\psi\in  \cD(\mu)\cap H^\infty: \psi g\in \cJ\}.
\] 
We have $ h \in \widetilde{\cJ}$, so {$ \underline{\cZ}(\widetilde{\cJ}) \subset V$. Also 
 $\mathcal{J}\cap H^\infty \subset \widetilde{\cJ}$, hence $\underline{\cZ}(\widetilde{\cJ})\subset \underline{\cZ} (\cJ)\cap V= \{\zeta_0\}$. }
Since $h$ is an  outer function, by Lemma \ref{lem68},   $(z-\zeta_0)\in \widetilde{\cJ}$. Thus 
$$(z-\zeta_0)g \phi\in [f]_{ \cD(\mu)}.$$ 
  As before we distinguish 2 cases.

\begin{itemize}
 \item If  $c_\mu (\{\zeta_0\})=0$, then  $z-\zeta_0$ is cyclic . Hence  $g\phi\in [f]_{  \cD(\mu)}.$

\item If  $c_\mu (\{\zeta_0\})>0$,   we distinguish again 2 cases.

\noindent - if $f(\zeta_0)=0$ then $\zeta_0\in E$, $\cM_\mu(E)\subset \cM_\mu(\{\zeta_0\})=$ and 
   \begin{eqnarray*}[(z-\zeta_0)\phi]_{  \cD(\mu)}&=&[(z-\zeta_0)]_{  \cD(\mu)} \cap [\phi]_{  \cD(\mu)}\\
   &=&\cM_\mu(E)\\
   &=&[\phi]_{\cD(\mu)}.
\end{eqnarray*}
 \noindent  So 
\begin{eqnarray*}
 [g\phi]_{  \cD(\mu)} 
 &=&[g]_{  \cD(\mu)}\cap [\phi]_{  \cD(\mu)}\\
 &=&[g]_{  \cD(\mu)}\cap [(z-\zeta_0)\phi]_{  \cD(\mu)}\\
 &=&[(z-\zeta_0)g\phi]_{  \cD(\mu)}\\
 &\subset& [f]_{  \cD(\mu)}.
 \end{eqnarray*}

\noindent - if $f(\zeta_0)\neq 0$ then $h(\zeta_0)=f_V(\zeta_0)\neq 0$ and $h\in \widetilde{\cJ}$. As before
\[
g\phi=\frac{h}{h(\zeta_0)}g\phi +\frac{h(\zeta_0)-h}{h(\zeta_0)}g\phi\in  [f]_{  \cD(\mu)} .
 \]

\end{itemize}

In all cases we have $g\phi\in  [f]_{  \cD(\mu)}$. We get,  as before in the proof of Theorem \ref{Th1}, 
that   $ \cD(\mu)\cap H^\infty= \cJ$ and 
\[
\mathcal{M}_{\mu}(E)\subset  [f]_{  \cD(\mu)}.
\]
Now the proof is complete.

\hfill $\Box$
\section{Final remarks}

\begin{itemize}
\item Note that if $\mu$ is a positive finite measure such that $\supp \mu \subset E_1\cup E_2$,
 where $E_1$ and $E_2$ are disjoint closed subsets of $\TT$, then $\cD (\mu)=\cD (\mu _{|E_1})\cap \cD (\mu _{|E_2})$. In this case every closed invariant subspace $\cM$ of $\cD (\mu)$ can be written as  
 $\cM = \cM _1\cap \cM _2$, where $\cM _i$ is an invariant subspace of  $ \cD (\mu _{|E_i})$ ($i=1,2$).\\

\item A closed set $E$ is union of a countable set and perfect set.  
Using the same argument as in the proof of Theorem \ref{Th1},  one can prove that 
 Brown-Shields conjecture is true for $\cD(\mu)$ if and only if  
 Brown-Shields conjecture is true for $\cD(\mu _{| \cP (\supp \mu )})$ 
 where $\cP (\supp \mu )$ is the perfect core of the support of $\mu$.\\

\end{itemize}

\subsection*{Acknowledgments} The authors are grateful to the referee for his valuable  remarks and suggestions.

\end{document}